\documentclass[12pt]{amsart}

\usepackage{amsmath}
\usepackage{amssymb}
\usepackage{amsthm}
\usepackage{mathrsfs}
\usepackage{comment}
\usepackage{hyperref}

\usepackage[all,cmtip]{xy}\usepackage{xcolor}
\usepackage{enumerate}
\usepackage{bm}
\usepackage{dsfont}
\usepackage{mathtools}
\usepackage[cal=euler]{mathalfa}
\usepackage[top=1in, bottom=1.25in, left=1.25in, right=1.25in]{geometry}
\usepackage{parskip}

\hypersetup{colorlinks=true,linkcolor=magenta,citecolor=blue}

\DeclareMathAlphabet{\mathpzc}{OT1}{pzc}{m}{it}

\theoremstyle{plain}
\newtheorem{theorem}{Theorem}[section]
\newtheorem*{theorem*}{Theorem}
\newtheorem{maintheorem}{Theorem}
\newtheorem{lemma}[theorem]{Lemma}
\newtheorem{proposition}[theorem]{Proposition}

\newtheorem{corollary}[theorem]{Corollary}

\theoremstyle{definition}
\newtheorem{definition}[theorem]{Definition}

\newtheorem{remark}[theorem]{Remark}

\numberwithin{equation}{section}
\numberwithin{figure}{section}

\newcommand{\val}{\mathrm{val}}

\newcommand{\Res}{\mathrm{Res}} 
\newcommand{\Ind}{\mathrm{Ind}} 
\newcommand{\cInd}{\textrm{c-}\mathrm{Ind}}
\newcommand{\Hom}{\mathrm{Hom}}

\newcommand{\Rep}{\mathrm{Rep}}

\newcommand{\Z}{\mathbb{Z}} 

\newcommand{\Lie}{\mathrm{Lie}}

\newcommand{\bG}{\mathbf{G}}
\newcommand{\bS}{\mathbf{S}}

\newcommand{\bZ}{\mathbf{Z}}

\newcommand{\bH}{\mathbf{H}}
\newcommand{\bT}{\mathbf{T}}

\newcommand{\apart}{\mathscr{A}}

\newcommand{\buil}{\mathscr{B}}

\newcommand{\Aut}{\mathrm{Aut}}

\begin{document}
\setlength{\parindent}{0pt}
\title{Restricting admissible representations to fixed-point subgroups}
\author{Peter Latham}
\address{Department of Mathematics and Statistics, University of Ottawa, Ottawa, Canada}
\email{platham@uottawa.ca}

\author{Monica Nevins}
\address{Department of Mathematics and Statistics, University of Ottawa, Ottawa, Canada}
\email{mnevins@uottawa.ca}
\thanks{The second author's research is supported by a Discovery Grant from NSERC Canada.}
\date{\today}

\begin{abstract}
\noindent Given a $p$-adic group $G=\bG(F)$ and a finite group $\Gamma\subset\mathrm{Aut}_F(\bG)$ such that the fixed-point subgroup $\bG^\Gamma$ is reductive, we establish a relationship between constructions of types for $G$ and $G^\Gamma$ due to Yu, Kim--Yu and Fintzen, which generalizes the relationship between general linear and classical groups observed by Stevens. As an application, given an irreducible representation $\pi$ of $G$, we explicitly identify a number of the inertial equivalence classes occurring in the representation $\pi|_{G^{[\Gamma]}}$.
\end{abstract}
\maketitle

Given a reductive $p$-adic group $G$, there are two distinct approaches to constructing types for $G$ (in the sense of \cite{BushnellKutzko1998}), each of which relies on differing hypotheses (and, whenever these two constructions overlap, they produce the same types). Firstly, one has the construction of Yu, Kim--Yu and Fintzen \cite{Yu2001,KimYu2017,Fintzen2021}, which is valid for an arbitrary tamely ramified group $G$, but which only produces types for every inertial equivalence class if one imposes additional hypotheses on the prime number $p$. This construction is notable in that it is uniform across all such $G$.

In contrast, the constructions of Bushnell--Kutzko, Stevens, and Blasco--Blondel, each of which tackles only a specific family of groups, come with the advantage of requiring no tameness assumptions on the prime $p$. In each, the key is the Bushnell--Kutzko construction of \emph{semisimple characters} for $\mathbf{GL}_n(F)$ \cite{BushnellKutzko1993,BushnellKutzko1999}, which are a family of characters of certain pro-$p$ subgroups of $\mathbf{GL}_n(F)$ satisfying a number of remarkable properties. The constructions of Stevens \cite{Stevens2005,Stevens2008}, and Blasco--Blondel \cite{BlascoBlondel2012} each rely on exploiting the relationships between $G$ and $\mathbf{GL}_n(F)$ in order to define a new family of semisimple characters by restriction of those of Bushnell--Kutzko. This approach by nature reveals a network of relationships between the theories of types for the various groups for which such a construction has been carried out. Nowadays such relationships fall under the umbrella term of \emph{endo-equivalence}, with the case of $\mathbf{GL}_n(F)$ and classical groups being just one particularly simple example of a more general phenomenon \cite{KurinczukSkodlerackStevens2021}.

Since the constructions of Yu \emph{et al.} treat each group in abstraction, such relationships among related groups do not immediately follow---while they do produce a family of semisimple characters (with a more straightforward construction than is seen in the Bushnell--Kutzko setting, due to the hypotheses on $p$), to date the only results relating these constructions have been between $G$ and its derived subgroup \cite{Bourgeois2020}. Our goal in this paper is to show that, in the setting of Yu \emph{et al.}'s construction, the relationship between semisimple characters for $\mathbf{GL}_n(F)$ and semisimple characters for classical groups admits a significant generalization.

Before stating our results, it is useful to quickly recap the approaches of \cite{Stevens2005} and \cite{BlascoBlondel2012}, which provide our motivating examples. In \cite{Stevens2005}, Stevens considers a classical group $G$, identified as the fixed points in $\tilde{G}=\mathbf{GL}_n(F)$ of a Galois involution $\gamma$. The Bushnell--Kutzko construction, along with a generalization due to Stevens in \emph{loc. cit.}, associates to each \emph{semisimple stratum} $(\frak{a},\beta)$ for $\tilde{G}$ a finite set $\tilde{\mathcal{C}}(\frak{a},\beta)$ of semisimple characters of a pro-$p$ open subgroup $\tilde{H}_+(\frak{a},\beta)$ of $\tilde{G}$. Stevens shows that if $(\frak{a},\beta)$ is \emph{skew}, \emph{i.e.} if $\frak{a}$ is $\gamma$-stable and $\beta$ is contained in the Lie algebra of $G$, then $\tilde{H}_+(\frak{a},\beta)$ is stable under the action of $\gamma$, in which case he is able to define $H_+(\frak{a},\beta)=\tilde{H}_+(\frak{a},\beta)^\gamma$ and equip this group with a set $\mathcal{C}(\frak{a},\beta)$ of semisimple characters by restriction. It transpires that the semisimple characters obtained in this way still have the desired properties, and this approach ultimately leads to a construction of types for classical groups \cite{Stevens2008}. Building upon this, in \cite{BlascoBlondel2012} Blondel and Blasco consider the case where $\mathbf{G}_2(F)$ is the fixed points under triality in $\mathbf{SO}_8(F)$; they establish the analogous relationships between semisimple characters, with this yielding a construction of supercuspidal representations of $\mathbf{G}_2(F)$.

The technical result underlying the present paper is a generalization of this relationship to the Yu setting, and to a wider class of groups. Specifically, we consider an arbitrary tamely ramified group $G$, equipped with a finite group $\Gamma$ of algebraic automorphisms, the order of which is coprime to $p$. Then (the connected component of) the fixed-point subgroup $G^\Gamma$ is a closed reductive subgroup of $G$ \cite{PrasadYu2002}. Here, the situation is rather different to that considered by Stevens, in that \emph{one already has} a construction of semisimple characters for both $G$ and $G^\Gamma$. Thus, rather than providing a new construction of types, our goal is to show that the existing constructions of types are related in a manner directly generalizing the situation of \cite{Stevens2005,BlascoBlondel2012}.

The construction of semisimple characters for the group $G$ proceeds by defining a family of \emph{truncated data} $\Sigma$.  To each such datum $\Sigma$, one may associate a pro-$p$  open subgroup $H_+(\Sigma)$ of $G$, together with a semisimple character $\theta_\Sigma$ of $H_+(\Sigma)$. We identify a family of truncated data which we call \emph{$\Gamma$-stable} that provide the natural generalization of the notion of \emph{skewness} considered in \cite{Stevens2005}, and prove the following result.

\begin{maintheorem}[{Theorem~\ref{thm:simplechars}}]
Suppose that $\Sigma$ is a $\Gamma$-stable truncated datum. Then the group $H_+(\Sigma)$ is stable under the action of $\Gamma$, and the pair $(H_+(\Sigma)^\Gamma,\theta_\Sigma)$ is a semisimple character for $G^\Gamma$.
\end{maintheorem}

In fact, this is explicit: we are able to write down a datum $\Sigma^\Gamma$ for $G^\Gamma$ such that $(H_+(\Sigma)^\Gamma,\theta_\Sigma)=(H_+(\Sigma^\Gamma),\theta_{\Sigma^\Gamma})$.

Given such a natural relationship between the semisimple characters of $G$ and $G^\Gamma$, it is natural to ask to what extent this relationship extends to one between the corresponding types. 
Specifically, given a type $(J,\lambda)$ for $G$ which contains some $\Gamma$-stable semisimple character $(H_+,\theta)$, we may ask for a description of the irreducible components of $\lambda|_{J\cap G^\Gamma}$. The situation here is rather less elegant, due to the fact that not every irreducible component $\xi$ of $\lambda|_{J\cap G^{[\Gamma]}}$ will be a type. Nonetheless, we are able to adapt the arguments of \cite[\S 7]{LathamNevins2020} in order to 
associate to each such 
$\xi$ an inertial equivalence class $\frak{s}_\xi$ of representations of $G^\Gamma$, and prove the following.

\begin{maintheorem}[{Theorem~\ref{thm:main}}]
Let $\pi$ be a representation of $G$ which is generated by its $\lambda$-isotypic vectors. For each irreducible component $\xi$ of $\lambda|_{J\cap G^\Gamma}$, the representation $\pi|_{G^\Gamma}$ contains an irreducible subquotient of inertial support $\frak{s}_\xi$.
\end{maintheorem}

These two results give a uniform framework for some important questions of branching to reductive subgroups that can otherwise only be treated case-by-case.  For example, taking $\Gamma$ to be a finite-order inner automorphism of $\bG$ yields information about branching to some Levi (and twisted Levi) subgroups $\bG^{\Gamma}$.  On the other hand, working with a group $\Gamma$ of outer automorphisms is a key tool in understanding representations of exceptional groups.

Several interesting questions remain open.  One is to classify the semisimple characters of $G^{\Gamma}$ obtained through this construction; one obstacle to lifting arbitrary truncated data from $G^{\Gamma}$ to $G$ relates to satisfying the genericity condition (iv) of Definition~\ref{def:td}.
 For another, the two hypotheses on $\Gamma$---that it is finite, and that its order is prime to $p$---are used     in the proofs of most of the results in Section~\ref{sec:gamma-stable}, through the work of \cite{PrasadYu2002} and \cite{Portilla2014}.  It would be particularly interesting to relax the first hypothesis, at the expense of some generality, to be able to say more about restrictions to more general subgroups of $G$.   

This paper is organized as follows.  We establish our notation in Section~\ref{S:notation} and define (truncated) data as well as their associated semisimple characters and types in Section~\ref{S:data}, making use of several results from \cite{Portilla2014} and \cite{Fintzen2021}. In Section~\ref{sec:gamma-stable} we define the notion of $\Gamma$-stable truncated data, and use these to prove Theorem~\ref{thm:simplechars} in Section~\ref{S:simplechars}.  We conclude in Section~\ref{S:HM} with an application to restrictions of representations, in the form of Theorem~\ref{thm:main}.

\section{Notation} \label{S:notation}

Let $F$ be a locally compact field, complete with respect to a non-archimedean discrete valuation $\nu$, with ring of integers $\frak{o}$, maximal ideal $\frak{p}$, and finite residue field $\frak{f}=\frak{o}/\frak{p}$ of cardinality $q=p^r$.  We fix a choice of additive character $\psi \colon F \to \mathbb{C}^\times$ that is trivial on $\frak{p}$ but non-trivial on $\frak{o}$.  Write $\mathrm{val}$ for a valuation of $F$ and, for any finite extension $E$ of $F$, denote also by $\mathrm{val}$ its unique $F$-normalized extension to $E$.   We let $\bar{F}$ denote a separable algebraic closure of $F$.

Let $\bG$ denote a connected reductive group defined over $F$, and  write $G=\bG(F)$ for its group of rational points, equipped with the locally profinite topology. We use similar notation for closed subgroups without comment. The Lie algebra of the algebraic group $\bG$ will be denoted by $\Lie(\bG)$, 
with its $F$-rational points being $\frak{g}=\Lie(\bG)(F)$; we denote the duals of these algebras by $\Lie(\bG)^*$ and $\frak{g}^*$, respectively.  For any suitable object $*$, write $\mathrm{Cent}_G(*)$ for the centralizer in $G$ of $*$ and $Z(*)$ for its centre.

We write $\buil(G)=\buil(\bG,F)$ for the (extended) Bruhat--Tits building of $G$. Given a point $x\in\buil(G)$, we write $G_x$ for the subgroup of $G$ which fixes $x$; this group carries the Moy--Prasad filtration $G_{x,r}$ for $r\geq 0$; the group $G_{x,0}$ is the parahoric subgroup associated to $x$, and for any $r>0$ the group $G_{x,r}$ is a pro-$p$ normal subgroup of $G_{x}$. For any $0<s<r$ we abbreviate $G_{x,s:r}=G_{x,s}/G_{x,r}$ and for any $r\geq 0$, we abbreviate $G_{x,r+}=\bigcup_{s>r}G_{x,s}$.
In fact, as established in \cite{Yu2015}, for $r\geq 0$ the groups $G_{x,r}$ are schematic, in the sense that there exist canonical smooth connected affine $\frak{o}$-models $\bG_{x,r}$ of $\bG$ such that $\bG_{x,r}(\frak{o})=G_{x,r}$.  As a special case, due to Bruhat and Tits, note that the special fibre of $\bG_{x,0}$ is a connected reductive $\frak{f}$-algebraic group whose group of $\frak{f}$ -points is $G_{x,0:0+}$.

 Similarly, for any $r\in \mathbb{R}$ and $x\in \buil(G)$ we have a filtration of $\frak{g}$ by $\frak{o}$-lattices $\frak{g}_{x,r}$ and $\frak{g}_{x,r+}$; this induces a corresponding filtration on $\frak{g}^*$ by setting $\frak{g}_{x,r}^*=\{X\in\frak{g}^*\ |\ \forall Y\in\frak{g}_{x,-r+},\ X(Y)\in\frak{p}\}$.  For any $0<s<r\leq 2s$ the Moy--Prasad isomorphism identifies $\frak{g}_{x,s:r}= \frak{g}_{x,s}/\frak{g}_{x,r}$ with $G_{x,s:r}$.  

We fix a finite subgroup $\Gamma\subset\Aut_F(\bG)$, and assume that $p$ is coprime to the order of $\Gamma$. Via the action of $\Gamma$ on $\bG$, we obtain natural actions of $\Gamma$ on $G$, on $\buil(G)$, on $\frak{g}$, and on $\frak{g}^*$. Given a $\Gamma$-set $X$, we write $X^\Gamma$ for the set of points $x\in X$ which are fixed by $\Gamma$. 

To ease notation, if $\mathbf{H}$ is a $\Gamma$-stable subgroup of $\bG$, then we write $\mathbf{H}^{[\Gamma]} = (\mathbf{H}^\Gamma)^\circ$ for the connected component of $\mathbf{H}^\Gamma$, and set $H^{[\Gamma]} = \mathbf{H}^{[\Gamma]}(F)$, to be distinguished from $H\cap G^{[\Gamma]}$ as necessary.

By  \cite{PrasadYu2002},  $\bG^{[\Gamma]}$ is a closed reductive subgroup of $\bG$, whence $G^{[\Gamma]}$ is a closed subgroup of $G$; moreover, there is a family of $G^{[\Gamma]}$-equivariant embeddings $\buil(G^{[\Gamma]})\hookrightarrow\buil(G)$, each of which identifies $\buil(G^{[\Gamma]})$ with $\buil(G)^\Gamma$. We fix, once and for all, such a choice of embedding.


Finally, throughout the paper we impose the hypothesis that $\bG$ splits over a tamely ramified extension of $F$, and that $p$ is coprime to the orders of the (absolute) Weyl groups of both $\bG$ and $\bG^{[\Gamma]}$. This hypothesis is required in order to apply the results of \cite{Fintzen2021}, and it ensures that each twisted Levi subgroup of $\bG$ that is defined over $F$ splits over a tamely ramified extension.


\section{Data and semisimple characters}\label{S:data}

In this section, we recall the fundamental notion of a datum for constructing a type, following the approach of \cite{Fintzen2021} (although with our indices running in the opposite direction, following the convention from the initial definition in \cite{Yu2001}).  We then introduce various groups and representations associated to a datum, 
choosing notation similar to that found in \cite[\S 5]{LathamNevins2020}.

\begin{definition} \label{def:td}
A \emph{truncated datum} for $G$ is a tuple $\Sigma=(\vec{G},x,\vec{r},\vec{X}, \vec{\phi})$ consisting of the following objects.
\begin{enumerate}[(i)]
\item A sequence $\vec{G}=(\bG^0\subsetneq \bG^1\subsetneq\cdots\subsetneq \bG^d\subseteq \bG^{d+1}=\bG)$ of twisted Levi subgroups of $\bG$, each of which is defined over $F$ (and splits over a tamely ramified extension of $F$).
\item A point $x \in \buil(G^0)\subset \buil(G^1)\subset \cdots \subset \buil(G)$, relative to a fixed choice of embeddings of buildings.
\item A sequence $\vec{r}=(0<r_0<r_1<\cdots<r_{d-1}<r_d)$ of real numbers; we set $s_i=r_i/2$ and $s_{-1}=0$ throughout.
\item A sequence $\vec{X} = (X^0, X^1, \ldots, X^{d})$ consisting of elements $X^i\in(\frak{g}^i)^*$ which are \emph{almost strongly stable and $\bG^{i+1}$-generic of depth $-r_i$  at $x \in \buil(G^{i+1})$}.  (We recall these terms from \cite[Def. 3.2, Def. 3.5]{Fintzen2021} in the next section.)  Further, for each $0\leq i\leq d$ we require $\mathrm{Cent}_{G^{i+1}}(X^i)=G^i$.
\item A sequence $\vec{\phi} = (\phi^0, \phi^1, \ldots, \phi^{d})$ of characters $\phi^i$ of $G^i$ of depth $r_i$ realized by $X^i$, that is, such that via the Moy-Prasad isomorphism $G^i_{x,s_i+:r_i+}\cong \frak{g}^i_{x,s_i+:r_i+}$, the restriction of $\phi^i$ to $G^i_{x,s_i+}$ is given by $\psi\circ X^i$.
\end{enumerate}
\end{definition}

\begin{remark}
Omitting (v) from the definition above gives the ``truncated extended data" of \cite{Fintzen2021}, which parametrize strictly larger collections of representations. 
It is convenient for us to include the sequence of characters $\vec{\phi}$ so that our truncated data directly parametrize semisimple characters, as below.
\end{remark}

\begin{definition}\label{def:datum}
A truncated datum $\Sigma$ is completed to a datum $\tilde{\Sigma}=(\Sigma;\sigma)$ if in addition to (i) through (v) above, we have
\begin{enumerate}
\item[(vi)] $\sigma$ is an irreducible representation of $G^0_{x}$ that is trivial on $G^0_{x,0+}$ and whose restriction to $G^0_{x,0}$ is a sum of cuspidal representations inflated from $G^0_{x,0:0+}$. 
\end{enumerate}
\end{definition}

\begin{remark}
Given a datum $\tilde{\Sigma}$, 
by omitting reference to the sequence $\vec{X}$ and letting $M^0$ denote the Levi subgroup of $G^0$ associated to the point $x$ by \cite[\S 6.3]{MoyPrasad1996}, one recovers a datum in the sense of \cite{KimYu2017}. A consequence of the main result of \cite{Fintzen2021} is that every irreducible representation of $G$ contains a datum satisfying (i)--(vi).
\end{remark}

Let us now define two key compact open subgroups of $G$, each equipped with an irreducible representation, derived from these data.   Note that in \cite{Yu2001}, our group $H_+(\Sigma)$ is denoted $K_+$ and our group $J(\Sigma)$ is denoted ${}^\circ K$.

First, to a truncated datum $\Sigma$ one may attach the pro-$p$ group
\begin{equation}\label{eq:H+}
H_+(\Sigma)=G_{x,0+}^0G_{x,s_{0}+}^1\cdots G_{x,s_{d}+}^{d+1}.
\end{equation}
To define the \emph{semisimple character} $\theta_\Sigma$ of $H_+(\Sigma)$, we proceed as follows.

For each $i$ such that $0\leq i \leq d$, let $\bZ = Z(\bG^i)^{\circ}$; then $Z=\bZ(F)$ decomposes $\frak{g}$ into a direct sum of isotypic subspaces such that $\frak{g}^{i}$ is the $0$-isotypic subspace.  Write $\frak{n}^i$ for the sum  of the remaining isotypic components. 
 
From the pair $(X^i,\phi^i)$ we construct a character $\hat{\phi}^i$ of $G^i_{x}G_{x,s_i+}$ following \cite[\S4]{Yu2001}:
since $X^i$ is of depth $-r_i$,  
 $\psi \circ X^i$ defines a character of $\frak{g}^{i}_{x,s_i+}$ that is trivial on $\frak{g}^i_{x,r_i+}$.  Define a character $\tilde{\phi}^i$ of $G_{x,s_i+}$ by extending $\psi \circ X^i$ trivially across 
  $\frak{n}^i\cap\frak{g}_{x,s_i+}$, and then using the Moy--Prasad isomorphism to lift this to the group.
Since $\tilde{\phi}^i$ and $\phi^i$ agree on $G^i_{x}\cap G_{x,s_i+}$
we may define, for any $g\in G^i_x$ and $h\in G_{x,s_i+}$, a character $\hat{\phi}^i$ of $G^i_xG_{x,s_i+}$ by
\begin{equation}\label{eq:phihat}
\hat{\phi}^i(gh) =  \phi^i(g)\tilde{\phi^i}(h).
\end{equation}
Write again $\hat{\phi}^i$ for its restriction to the group $H_+(\Sigma) \subset G^i_xG_{x,s_i+}$.  The semisimple character $\theta_\Sigma$ is then defined as
$\theta_\Sigma=\prod_{i=0}^d\hat{\phi}^i.$

Next, to $\Sigma$ one may also associate the group
\[J(\Sigma)=G^0_{x}G_{x,s_0}^1\cdots G_{x,s_{d}}^{d+1},
\]
which contains $H_+(\Sigma)$ as an open normal subgroup. The semisimple character $\theta_\Sigma$ admits a \emph{Heisenberg--Weil lift} to $J(\Sigma)$.  This is a certain irreducible representation $\kappa_\Sigma$ of $J(\Sigma)$ which restricts to $H_+(\Sigma)$ as a sum of copies of $\theta_\Sigma$; we refer to \cite[\S 2.3, 2.4]{HakimMurnaghan2008} for details and will briefly recall the construction where needed in Section~\ref{S:HM}. 


If we now complete $\Sigma$ to a datum $\tilde{\Sigma}$ by choosing a representation $\sigma$, then we may also define a representation $\lambda_{\tilde{\Sigma}}=\sigma\otimes\kappa_\Sigma$ of $J(\Sigma)$. The fundamental fact  is that the pair $(J(\Sigma),\lambda_{\tilde{\Sigma}})$ is an $\frak{s}$-type, for some inertial equivalence class $\frak{s}=\frak{s}_{\Sigma}$ of $G$ 
(in the sense of \cite{BushnellKutzko1998}). Moreover, under our hypotheses on $p$, every irreducible representation of $G$ contains a type of the form $(J(\Sigma),\lambda_{\tilde{\Sigma}})$ for some datum ${\tilde{\Sigma}}$ \cite{Fintzen2021}.

\section{$\Gamma$-stable truncated data}\label{sec:gamma-stable}

Let $\Gamma \subset \mathrm{Aut}_F(\bG)$ be as in Section~\ref{S:notation}.  We begin with some straightforward lemmas.

\begin{lemma}\label{lem:tori}
Let $\bS\subset\bG^{[\Gamma]}$ be a maximal $F$-torus of $\bG^{[\Gamma]}$. Then $\bT = \mathrm{Cent}_{\bG}(\bS)$ is a $\Gamma$-stable maximal $F$-torus of $\bG$.  
\end{lemma}


\begin{proof}
Let $\bH=(\mathrm{Cent}_\bG(\bS))^\circ$. Then $\bH$ is a twisted Levi subgroup of $\bG$, hence a closed connected reductive $F$-subgroup. Since $\bS$ is fixed pointwise by $\Gamma$, $\bH$ is $\Gamma$-stable. By \cite{PrasadYu2002}, $\bH^{[\Gamma]}$ is a connected reductive group containing $\bS$ in its centre.

Suppose that the derived group $\bH'=[\bH,\bH]$ were non-trivial. This group is $\Gamma$-stable, and so again $(\bH')^{[\Gamma]}$ is a (non-trivial) reductive group. Choose a non-trivial $F$-torus $\bS'\subset(\bH')^{[\Gamma]}$. Then $\bS'$ commutes with $\bS$ and $\bS'\cap\bS$ is finite, so that the torus $\bS'\bS$ is of rank strictly larger than $\bS$ and yet contained in $\bH^{[\Gamma]}\subset\bG^{[\Gamma]}$, which contradicts the maximality of $\bS$. Therefore $\bH=\bT$ is a $\Gamma$-stable torus. 
\end{proof}

\begin{remark}\label{rem:split}
Note that since $p$ is coprime to the orders of the absolute Weyl groups of both $\bG$ and $\bG^{[\Gamma]}$, and $\bG$ splits over a tamely ramified extension of $F$, by \cite[Lem. 2.2]{Fintzen2021} this implies that every $F$-rational twisted Levi subgroup of $\bG$ or of $\bG^{[\Gamma]}$ splits over a tamely ramified extension of $F$.  Thus the tori $\bS$ and $\bT$ of the lemma split over a tamely ramified extension, but as remarked by \cite{Portilla2014}, it is not generally true that $\bT$  splits over $F$  even if $\bS$ does.
\end{remark}

Continuing in the setting of Lemma~\ref{lem:tori}, let $\Phi(\bG,\bT)$ denote the set of (absolute) roots of $\bG$ relative to $\bT$ and 
$\Phi(\bG^{[\Gamma]},\bS)$ the set of (absolute) roots of $\bG^{[\Gamma]}$ relative to $\bS$.  The elements $a$ of $\Phi(\bG^{[\Gamma]},\bS)$ occur as restrictions to $\bS$ of certain roots of $\bG$ relative to $\bT$.  
To describe these, note that the action of $\Gamma$ on $\bT$ and $\bG$ induces an action on the set of roots, and in consequence for each root $\alpha$ a linear action of the stabilizer $\Gamma_\alpha$ on the root space $\Lie(\bG)_\alpha$.  

\begin{lemma}\label{lem:roots}
The roots of $\bG^{[\Gamma]}$ relative to $\bS$ are the restrictions of those roots $\alpha$ of $\bG$ relative to $\bT$ for which $\Gamma_\alpha$ admits a nonzero fixed vector on $\Lie(\bG)_\alpha$, that is,
$$
\Phi(\bG^{[\Gamma]},\bS)=\{ \alpha |_{\bS} \ | \ \alpha \in \Phi(\bG,\bT), \ \Lie(\bG)_\alpha^{\Gamma_{\alpha}} \neq \{0\} \}.
$$
\end{lemma}

\begin{proof}
Let $\Gamma \cdot \alpha$ denote the orbit of a root $\alpha \in \Phi(\bG,\bT)$ and let $a=\alpha|_{\bS}$ be the common restriction to $\bS$.  The action of $\Gamma$ on $\oplus_{\beta \in \Gamma \cdot \alpha}\Lie(\bG)_\beta$ is induced from the action of $\Gamma_\alpha$ on $\Lie(\bG)_\alpha$, implying the result.
In particular, for any nonzero vector $Y$  in $\Lie(\bG)_\alpha^{\Gamma_{\alpha}}$, taking the average over its $\Gamma$-orbit 
$
\overline{Y}=\frac{1}{[\Gamma:\Gamma_\alpha]}\sum_{\gamma\in \Gamma/\Gamma_\alpha} \gamma\cdot Y
$
yields a $\Gamma$-fixed sum of linearly independent vectors, so a nonzero vector in $\Lie(\bG^{[\Gamma]})_{a}$.
\end{proof}


Our next lemmas relate to $\Lie(G)^*$.

\begin{lemma}\label{lem:dual-embedding}
There is a canonical identification $\Lie(\bG^{[\Gamma]})^*=(\Lie(\bG)^*)^\Gamma$.
\end{lemma}

\begin{proof}
We have $\Lie(\bG)^\Gamma = \Lie(\bG^{[\Gamma]})$.
Since $p$ does not divide $|\Gamma|$, Maschke's theorem applies, so that both $\Lie(\bG)$ and $\Lie(\bG)^*$ are completely decomposable as $\Gamma$-modules.  Thus there is a natural $F$-linear isomorphism $\Lie(\bG^{[\Gamma]})^*\rightarrow(\Lie(\bG)^*)^\Gamma$ defined by extending any functional $f$ on $\Lie(G^{[\Gamma]})$ by zero on the unique $\Gamma$-invariant complement in $\Lie(\bG)$; its inverse is given by restriction.
\end{proof}

We write $\frak{g}^\Gamma$ for $\Lie(\bG)^\Gamma(F)=\Lie(\bG^{[\Gamma]})(F)$, and similarly for the dual.

Given a nonzero element $X\in\frak{g}^*$ and a point $x\in\buil(G)$,  the \emph{depth of $X$ at $x$} is defined to be $d_G(x,X)=\mathrm{max}\{r\ |\  X\in\frak{g}_{x,r}^*\}$. The \emph{depth} of $X$ is $d_G(X)=\mathrm{sup}\{d_G(x,X)\ |\ x\in\buil(G)\}$, when the supremum exists.

\begin{lemma} \label{lem:d}
Suppose that $X\in\frak{g}^*\setminus \{0\}$ and $x\in\buil(G)$ are both $\Gamma$-fixed. Then, viewing $X$ as an element of $(\frak{g}^\Gamma)^*$ and $x$ as a point in $\buil(G^{[\Gamma]})$, we have $d_G(x,X)=d_{G^{[\Gamma]}}(x,X)$. 
\end{lemma}

\begin{proof}
Taking $\Gamma$-fixed points preserves Moy--Prasad filtrations on the Lie algebra, by \cite[Lem. 2.8]{Portilla2014}. That is, for all $r\in \mathbb{R}$ and $x\in\buil(G^{[\Gamma]})$, we have $(\frak{g}^\Gamma)_{x,r}=\frak{g}^\Gamma\cap\frak{g}_{x,r}$. Since $\frak{g}_{x,r}^*$ is defined as $\{X\in\frak{g}^*\ |\ \forall Y\in\frak{g}_{x,-r+},\ X(Y)\in\frak{p}\}$, taking $\Gamma$-fixed points preserves the filtration on the dual as well.
\end{proof}


We now define almost stable and almost strongly stable elements, following \cite[Def. 3.1]{Fintzen2021}.

We say an element $X\in\frak{g}^*$ is \emph{almost stable} if its $\bG$-orbit in $\Lie(\bG)^*$ is closed.  Note that if we identify $\frak{g}^*$ and $\frak{g}$ via a non-degenerate $G$-equivariant bilinear form, then almost stable elements correspond to semisimple elements of the Lie algebra.  In particular, by \cite[Lemma 3.3.4]{AdlerDeBacker2002}, such elements have finite depth.

Now suppose $X$ is almost stable and let $r=d(x,X)$; then the corresponding quotient element $\overline{X} \in \frak{g}^*_{x,r:r+}$ is nonzero.  We say that  $X$ is \emph{almost strongly stable} at $x$ if the orbit of $\overline{X}$ in $\frak{g}^*_{x,r:r+}$ under the special fibre of the $\frak{o}$-scheme $\bG_{x,0}$  does not contain $0$ in its Zariski closure.  This is the condition that the coset
$X+\frak{g}^*_{x,r+}$ is nondegenerate, in the sense of \cite{AdlerDeBacker2002}; for $p$ sufficiently large (and again identifying $\frak{g}^*$ and $\frak{g}$), it is equivalent to the statement that this coset does not contain any nilpotent elements.

\begin{lemma} \label{lem:as}
Suppose that $X\in\frak{g}^*$ is almost stable under the coadjoint action of $\bG$ and that $\Gamma\cdot X=X$. Then, viewed as an element of $(\frak{g}^\Gamma)^*$, $X$ is almost stable under the coadjoint action of $\bG^{[\Gamma]}$.
\end{lemma}

\begin{proof}
Since $X$ is $\Gamma$-fixed, it lies in $(\frak{g}^*)^\Gamma$. Since $\bG^{[\Gamma]}$ is a closed subgroup of $\bG$, the orbit of $X$ under $\bG^{[\Gamma]}$ is also closed, and is contained in the $\bG^{[\Gamma]}$-invariant subspace $(\Lie(\bG)^*)^\Gamma\simeq \Lie(\bG^{[\Gamma]})^*$.
\end{proof}

\begin{lemma}\label{lem:twisted-levis}
Suppose that $X\in\frak{g}^*$ is almost stable under the coadjoint action of $\bG$, and that $\Gamma\cdot X=X$. Let $\bG'=\mathrm{Cent}_\bG(X)$. Then $(\bG')^{[\Gamma]}=\mathrm{Cent}_{\bG^{[\Gamma]}}(X)$.
In particular, $(\bG')^{[\Gamma]}$ is a twisted Levi subgroup of $\bG^{[\Gamma]}$.
\end{lemma}


\begin{proof}
From the identification $(\frak{g}^\Gamma)^*=(\frak{g}^*)^\Gamma$, one sees that $\bG'\cap\bG^{[\Gamma]}=\mathrm{Cent}_{\bG^{[\Gamma]}}(X)$.  Since by Lemma~\ref{lem:as}, $X$ is almost stable relative to $\bG^{[\Gamma]}$, we have by \cite[Lem. 3.4]{Fintzen2021} that this intersection is a twisted Levi subgroup of $\bG^{[\Gamma]}$.  Moreover, as recapped in \emph{loc.~cit.}, the hypotheses on $p$ ensure it is therefore connected and thus equal to $(\bG')^{[\Gamma]}$.  
\end{proof}

We have the inclusion $Z(\bG')^{[\Gamma]}\subseteq Z({\bG'}^{[\Gamma]})^\circ$, but note that equality should fail in general (\emph{cf.} Lemma~\ref{lem:roots}).

We now recall the notion of \emph{genericity}, as defined in \cite[Def. 3.5]{Fintzen2021}. Let $X\in\frak{g}^*$ be an almost stable element and let $\bG'=\mathrm{Cent}_{\bG}(X)$. Let $\bT$ be a maximal $F$-torus contained in $\bG'$ that splits over a tamely ramified extension $E$ of $F$. Let $\apart(\bG,\bT,E)$, respectively $\apart(\bG',\bT,E)$, denote the apartment of $\buil(\bG,E)$, respectively $\buil(\bG',E)$, corresponding to $\bT$.  Let $\Phi=\Phi(\bG,\bT)$ and $\Phi'=\Phi(\bG',\bT)$ be the corresponding root systems. For each $\alpha\in\Phi$, let $H_\alpha=d\alpha(1)\in\mathrm{Lie}(\bT)$ denote the corresponding coroot. By construction, we have $X(H_\alpha)=0$ for each $\alpha\in\Phi'$. 

\begin{definition}
Let $X\in \frak{g}^*$ be an almost stable element.  Choose a maximal $F$-torus $\bT$ in $\bG'=\mathrm{Cent}_{\bG}(X)$ and  $x\in\apart(\bG',\bT,E)\cap\buil(\bG',F)\subset\buil(\bG,F)$. We say $X$ is \emph{$\bG$-generic of depth $r$ relative to $x$} if $d(x,X)=r$ 
and furthermore $\mathrm{val}(X(H_\alpha))=r$ for each $\alpha\in\Phi(\bG,\bT)\setminus\Phi(\bG',\bT)$.  
\end{definition}

Note that if $\bG=\bG'$ then the condition in the definition is only that $d(x,X)=r$.  

\begin{lemma}\label{lem:sum}
Suppose two almost stable elements $X$ and $Y$ are $\bG$-generic of depth 
$r$ and $s$, respectively, relative to a point $x$, and $r<s$.  If $\mathrm{Cent}_{\bG}(X)=\mathrm{Cent}_{\bG}(Y)$ then $X+Y$ is almost stable and $\bG$-generic of depth $r$ relative to $x$.
\end{lemma}

\begin{proof}
Let
 $\bG' = \mathrm{Cent}_{\bG}(X)\subset \mathrm{Cent}_{\bG}(X+Y)$. Under the $\bG$-equivariant identification of $\Lie(\bG)$ with $\Lie(\bG)^*$, the almost stable elements $X$ and $Y$ correspond to  semisimple elements of the centre of $\Lie(\bG')$, so their sum is again almost stable.  Thus by \cite[Lemma 3.4]{Fintzen2021} the centralizer of $X+Y$ is again connected.  Since $d(x,X+Y)=\min\{d(x,X),d(x,Y)\}=r$, we have $\mathrm{val}((X+Y)(H_\alpha))=\mathrm{val}(X(H_\alpha))=r$ for all $\alpha\in\Phi(\bG,\bT)\setminus\Phi(\bG',\bT)$.  Thus in particular $\mathrm{Cent}_{\bG}(X+Y)\subset \bG'$ and genericity follows.
\end{proof}

\begin{lemma}\label{lem:ass}
Suppose that $X\in\frak{g}^*$ is almost strongly stable and $\bG$-generic of depth $r$ at $x\in\buil(G)$, and that both $x$ and $X$ are $\Gamma$-fixed. Then $X$ is almost strongly stable and $\bG^{[\Gamma]}$-generic of depth $r$ at $x\in\buil(G^{[\Gamma]})$.
\end{lemma}

\begin{proof}
Set $\bG'=\mathrm{Cent}_{\bG}(X)$; this is $\Gamma$-stable.  Since $X$ is almost stable and generic of depth $r$ at $x$, by \cite[Lemma 3.11]{Fintzen2021} we have $x \in \buil(G')$.  In fact $x\in \buil(G')^\Gamma=\buil({G'}^{[\Gamma]})$ and $X$, viewed as an element of $(\frak{g}^\Gamma)^*$, also has depth $r$ relative to $x$, by Lemma~\ref{lem:d}.  To carry out the proof we now make a choice of torus and move to a more suitable point $y$, as follows.

Let $\bS$ be a maximal $\Gamma$-fixed torus in $\bG'$ and $\bT$ its centralizer, which by Lemma~\ref{lem:tori} is a $\Gamma$-stable maximal torus.  Let $E$ be a tamely ramified extension of $F$ over which $\bT$ splits, and choose a point 
$y\in\apart({\bG'}^{[\Gamma]},\bS, E)\cap\buil({G'}^{[\Gamma]})\subset\apart(\bG',\bT,E)\cap \buil(G')$. 
By \cite[Lemma 4.4]{Fintzen2021}, $X$ is also almost strongly stable and generic of depth $r$ at $y$, so in particular $X \in (\frak{g}^*_{y,r})^\Gamma$.

Since $X$ is $\bG$-generic of depth $r$ relative to $y$, we have that $X(H_\alpha)=0$ for all $\alpha \in \Phi(\bG',\bT)$ and that $\val(X(H_\alpha)) = r$ for all $\alpha \in \Phi(\bG,\bT)\setminus \Phi(\bG',\bT)$.  We next show these conditions hold for the corresponding roots in $\bG^{[\Gamma]}$.
Let $a\in \Phi(\bG^{[\Gamma]},\bS)$ 
and choose  $\alpha \in \Phi(\bG,\bT)$ as in Lemma~\ref{lem:roots} such that $\alpha|_{\bS}=a$.  
Then the coroot $H_a \in \Lie(\bS)$ is a $\Gamma$-fixed vector in the span of $\{H_{\gamma\cdot\alpha} \mid \gamma \in \Gamma/\Gamma_\alpha\}$.  It follows that it must be a (nonzero) multiple of $\tilde{H}=\sum_{\gamma \in \Gamma/\Gamma_\alpha} H_{\gamma\cdot\alpha}$.  In particular, $ \alpha(\tilde{H}) \in \Z$ is nonzero and so  
 $H_a = 2 (\alpha(\tilde{H}))^{-1}\tilde{H}$.
Thus, since $X$ is $\Gamma$-fixed, we may evaluate
$$
X(H_a) = 2 (\alpha(\tilde{H}))^{-1}\sum_{\gamma \in \Gamma/\Gamma_\alpha}X(H_{\gamma\cdot\alpha}) = \frac{2\vert \Gamma/\Gamma_\alpha \vert}{\alpha(\tilde{H})} X(H_\alpha).
$$
By  Lemma~\ref{lem:d}, $X$ has depth $r$ at $y$.  If $a \in \Phi({\bG'}^{[\Gamma]},\bS)$, we have $X(H_a)=0$.   If instead
$a \in \Phi({\bG}^{[\Gamma]},\bS)\setminus \Phi({\bG'}^{[\Gamma]},\bS)$ then since $H_a \in \frak{g}^{\Gamma}_{y,0}$ we have the inequality $\val(X(H_a))\geq r$.   Since $\val(X(H_\alpha))=r$ and $p$ does not divide $2\vert \Gamma \vert$ by hypothesis, we conclude from the preceding that $\val(X(H_a))=r$.
  Thus $X$ is $\bG^{[\Gamma]}$-generic of depth $r$ relative to $y$.

Since by Lemma~\ref{lem:as} $X$ is almost stable relative to $\bG^{[\Gamma]}$, applying \cite[Lemma 3.8]{Fintzen2021} yields that $X$ is almost strongly stable at $y$.  Finally, since both $x$ and $y$ lie in $\buil({G'}^{[\Gamma]})$, we conclude that $X$ is almost strongly stable and $\bG^{[\Gamma]}$-generic of depth $r$ relative to $x$.
\end{proof}

With this preparation, we are now able to give a straightforward formulation of the class of truncated data of $G$ that we are interested in.

\begin{definition}
Let $\Sigma=(\vec{G},x,\vec{r},\vec{X},\vec{\phi})$ be a truncated datum for $G$. We say that $\Sigma$ is \emph{$\Gamma$-stable} if $\Gamma$ fixes both the point $x$ and each of the $X^i\in \frak{g}^*$.
%
\end{definition}


\begin{proposition}\label{prop:gamma-data}
Let $\Sigma=(\vec{G},x,\vec{r},\vec{X},\vec{\phi})$ be a $\Gamma$-stable truncated datum for $G$. Define a tuple $\Sigma^\Gamma_{\mathrm{ext}}=(\vec{G}^{[\Gamma]},x,\vec{r},\vec{X},\vec{\phi})$, where:
\begin{enumerate}[(i)]
\item $\vec{G}^{[\Gamma]}$ is the twisted Levi sequence in $\bG^{[\Gamma]}$ consisting of the groups $\bG^{i,[\Gamma]}=(\bG^i)^{[\Gamma]}$;
\item each element $X^{i}\in((\frak{g}^i)^*)^\Gamma$ is viewed as a element of $(\frak{g}^{i,[\Gamma]})^*$ by restriction; and
\item each character $\phi^i$ is viewed as a character of $G^{i,[\Gamma]}$ by restriction.
\end{enumerate}
For each $i$ such that there exists $k>0$ for which $\bG^{i,[\Gamma]}=\bG^{i+k,[\Gamma]}$, replace the tuple of subsequences $((\bG^{i,[\Gamma]},\ldots, \bG^{i+k,[\Gamma]}), (r_i, \ldots, r_{i+k}), (X^i, \ldots, X^{i+k}), (\phi^i, \ldots, \phi^{i+k}))$ with the tuple $(\bG^{i,[\Gamma]}, r_{i+k}, \sum_{j=i}^{i+k} X^j, \prod_{j=i}^{i+k}\phi^j)$.
Then the resulting tuple $\Sigma^\Gamma$ is a truncated datum for $G^{[\Gamma]}$.
\end{proposition}

\begin{proof}
We verify the conditions of Definition~\ref{def:td}.  By Lemma~\ref{lem:twisted-levis} and Remark~\ref{rem:split}, (i) holds for $\Sigma^\Gamma_{\mathrm{ext}}$ and thus for $\Sigma^\Gamma$, with the exception that the inclusions may not be strict for the former tuple.  Conditions (ii) and (iii) are given and (v) is immediate. By Lemma~\ref{lem:ass}, (iv) holds for $\Sigma^\Gamma_{\mathrm{ext}}$ and by Lemma~\ref{lem:sum}, recalling that each $X^i$ is of depth $-r_i$, it thus holds for $\Sigma^\Gamma$.    
\end{proof}

The difference between $\Sigma^\Gamma_{\mathrm{ext}}$ and $\Sigma^\Gamma$ is effectively only cosmetic.  For the purposes of defining the pairs $(H(\Sigma^\Gamma),\theta_{\Sigma^\Gamma})$ and $(J(\Sigma^\Gamma),\kappa_{\Sigma^\Gamma})$, we may use $\Sigma^\Gamma_{\mathrm{ext}}$ and the truncated datum $\Sigma^\Gamma$ interchangeably, as we make precise in the following sections.  In all cases we retain and refer only to the indices $i$ of the original datum.

%
%

\section{Semisimple characters associated to $\Gamma$-stable truncated data}\label{S:simplechars}

We now show that the restriction of truncated data corresponds to a restriction of semisimple characters.  We begin with some results about the compatibility of Moy--Prasad filtration subgroups.

\begin{lemma}\label{lem:gxrs}
Let $x\in\buil(G^{[\Gamma]})=\buil(G)^\Gamma$. For all $r>0$, one has
\[G_{x,r}\cap G^{[\Gamma]}=(G^{[\Gamma]})_{x,r}.
\]
\end{lemma}

\begin{proof}
Since $\Gamma\subset\Aut_F(\bG)$ and $\Gamma$ fixes $x$, we have that $G_{x,r}$ is $\Gamma$-stable (see, for example, \cite[Lem. 2.2.1]{AdlerDeBacker2002}). 
Since $x\in\buil(G)^\Gamma$ and $p$ does not divide $|\Gamma|$, it follows that in fact $\Gamma \subset \Aut_{\frak{o}}(\bG_{x,r})$.  Thus 
\[\bG_{x,r}\cap\bG^{[\Gamma]}=(\bG_{x,r})^{[\Gamma]}.
\]
By \cite[Lemma 2.8]{Portilla2014} we have $\frak{g}_{x,r}\cap \frak{g}^\Gamma = (\frak{g}^{\Gamma})_{x,r}=(\frak{g}^{[\Gamma]})_{x,r}$, whence the $\frak{o}$-Lie algebra of $(\bG_{x,r})^{[\Gamma]}$ coincides with that of $(\bG^{[\Gamma]})_{x,r}$; the result follows by connectedness.
\end{proof}
%
%

\begin{remark}
During the proof of \cite[Lem. 2.7]{Portilla2014}, Portilla shows that in the setting of the above lemma we have
\[(G^{[\Gamma]})_{x,0}=G_{x,0} \cap G^{[\Gamma]};
\]
he includes the proof that $(G^{[\Gamma]})_{x,0+}=G_{x,0+} \cap G^{[\Gamma]}$ as \cite[Prop. 5.1]{Portilla2014}.
\end{remark}

Given a $\Gamma$-stable datum $\Sigma$, each twisted Levi subgroup $\bG^i$ occuring in $\Sigma$ is $\Gamma$-stable, so by Lemma~\ref{lem:gxrs} and the remark following, for each $0\leq i\leq d+1$ and any $s\geq 0$ we have
\[((G^i)^{[\Gamma]})_{x,s}=(G_{x,s}^i)^{[\Gamma]};
\]
we write $G_{x,s}^{i,[\Gamma]}$ for this group.

\begin{lemma}\label{lem:thm1}
Let $G'$ be a $\Gamma$-stable twisted Levi subgroup of $G$ and $x\in\buil(G')^\Gamma\subset\buil(G)$. For any $0<t<s$ let $K$ be a $\Gamma$-stable subgroup of $G_{x,s}$ normalized by $G'_{x,t}$ and containing $G'_{x,s}$.  Then we have
\[
(G_{x,t}'K)^{\Gamma}= {G'}^{[\Gamma]}_{x,t}K^\Gamma.  
\]
\end{lemma}

\begin{proof}
Evidently ${G'}^{[\Gamma]}_{x,t}K^\Gamma\subseteq (G_{x,t}'K)^{\Gamma}$.  If $t=s$ then as $G'_{x,s}\subseteq K$ the result is immediate.  
Fix $t \in (0,s)$ and suppose the result has been shown for all $r$ satisfying $t<r\leq s$; this amounts to a finite set of distinct equalities since the filtrations are discrete.

Let $h=gk\in(G_{x,t}'K)^{\Gamma}$ with $g\in G_{x,t}'$ and $k\in K$. Then for each $\gamma\in\Gamma$ we have that $\gamma(gk)=gk$, and therefore
\[g^{-1}\gamma(g)=k\gamma(k^{-1})\in G_{x,t}'\cap K= G_{x,s}'.
\]
Consequently, for all $\gamma\in\Gamma$ we have
\[g\equiv\gamma(g)\quad\text{mod }G_{x,s}'.
\]
Now let $\bar{g}$ denote the image of $g$ in $G_{x,t:t+}'$; since $s>t$ the preceding implies that $\bar{g}\in(G_{x,t:t+}')^{\Gamma}$.  
By the proof of \cite[Lem. 2.28]{Portilla2014}, we have a natural isomorphism
\[(\frak{g}_{x,t:t+}')^{\Gamma}\simeq(\frak{g}'^{[\Gamma]})_{x,t:t+};
\]
whence via the Moy--Prasad isomorphism we may choose  a representative $g'$ of $\bar{g}$ satisfying $g'\in (G'^{[\Gamma]})_{x,t}$.  By Lemma~\ref{lem:gxrs} we have $g'\in (G'_{x,t})^\Gamma$.  Thus  $(g')^{-1}gk\in (G_{x,t+}'K)^{\Gamma}$, which by induction is equal to $G_{x,t+}'^{[\Gamma]}K^\Gamma$.  We conclude that
$gk \in g'G_{x,t+}'^{[\Gamma]}K^\Gamma\subset G_{x,t}'^{[\Gamma]}K^\Gamma$, as required.
\end{proof}

\begin{theorem}\label{thm:simplechars}
Suppose that $\Sigma$ is a $\Gamma$-stable truncated datum for $G$ and
let $\Sigma^\Gamma$ be the corresponding truncated datum for $G^{[\Gamma]}$. Then one has an equality of semisimple characters
\[(H_+(\Sigma^\Gamma),\theta_{\Sigma^\Gamma})=(H_+(\Sigma)^\Gamma,\theta_\Sigma|_{H_+(\Sigma)^\Gamma}).
\]
\end{theorem}

\begin{proof}
Let $\Sigma = (\vec{G},x,\vec{r},\vec{X}, \vec{\phi})$.  Recall that 
$H_+(\Sigma)=G^0_{x,0+}G^1_{x,s_0+}\cdots G^{d+1}_{x,s_d+}$.
Inductively applying Lemma~\ref{lem:thm1} yields the equality  
\[
H_+(\Sigma)^\Gamma = G^{0,[\Gamma]}_{x,0+}G^{1,[\Gamma]}_{x,s_0+}\cdots G^{d+1,[\Gamma]}_{x,s_d+},
\]
which is the group $H_+(\Sigma^\Gamma_{\mathrm{ext}})$ as defined in  \eqref{eq:H+}.
Since $0<s_0<\cdots < s_d$, if $G^{i,[\Gamma]}=G^{i+1,[\Gamma]}=\cdots = G^{i+k,[\Gamma]}$, then 
$$
G^{i,[\Gamma]}_{x,s_{i-1+}}G^{i+1,[\Gamma]}_{x,s_i+}\cdots G^{i+k,[\Gamma]}_{x,s_{i+k-1}+} = G^{i,[\Gamma]}_{x,s_{i-1}+}.
$$
Thus we have $H(\Sigma^\Gamma)=H(\Sigma^\Gamma_{\mathrm{ext}})=H_+(\Sigma)^\Gamma$.

For each $0\leq i \leq d$, we now compare the characters $\hat{\phi}^i$ defined in \eqref{eq:phihat} produced from these various tuples.

Let $\bZ^i = Z(\bG^i)^\circ$ and $\bZ^i_{\Gamma}=
Z(\bG^{i,[\Gamma]})^\circ$ and 
as usual let $Z^i=\bZ^i(F)$ and $Z^i_{\Gamma}=\bZ^i_{\Gamma}(F)$.
Then under $Z^i$ we have the isotypic decomposition $\frak{g} = \frak{g}^i\oplus \frak{n}$ where $\frak{g}^i$ is the $1$-isotypic subspace.  Since $Z^i$ is $\Gamma$-stable, so is $\frak{n}$, and taking $\Gamma$-fixed points yields a decomposition $\frak{g}^{[\Gamma]} = \frak{g}^{i,[\Gamma]} \oplus \frak{n}^\Gamma$.  Since  $\frak{g}^{i,[\Gamma]}$ is the $1$-isotypic component of $\frak{g}^{[\Gamma]}$ under  $Z^i_{\Gamma}$, this $Z^i_{\Gamma}$-invariant decomposition is isotypic.  It follows that  $\tilde{\phi^i}$, the character of $G_{x,s_i+}$ arising from the trivial extension across $\frak{n}$ of $\psi\circ X^i$, restricts to the extension $\tilde{\phi}^{i,\Gamma}$ of $\psi\circ X^i$ to $G^{[\Gamma]}_{x,s_i+}$.  Consequently, the character $\hat{\phi}^i$ of $G^i_xG_{x,s_i+}$ restricts to the corresponding character $\hat{\phi}^{i,[\Gamma]}$ of $G^{i,[\Gamma]}_xG^{[\Gamma]}_{x,s_i+}$.

Now suppose that in the above setting we have $G^{i,[\Gamma]}=G^{i+1,[\Gamma]}=\cdots = G^{i+k,[\Gamma]}$ and have therefore in $\Sigma^\Gamma$ replaced $(X^i, \cdots, X^{i+k})$ with $\sum_{j=i}^{i+k}X^j$.  It suffices to note that the restriction to $\frak{g}^{[\Gamma]}_{x,s_{i+k}+}$ of $\psi\circ \sum_{j=i}^{i+k}X^j = \prod_{j=i}^{i+k}\psi\circ X^j$ lifts to a well-defined character of this smallest subgroup $G^{[\Gamma]}_{x,s_{i+k}+}$, and that it moreover agrees with $\prod_{j=i}^{i+k}\phi^j$ on the intersection $G^{i,[\Gamma]}_x\cap G^{[\Gamma]}_{x,s_{i+k}+}$.  Thus in the notation of \eqref{eq:phihat}, we have
$$
\prod_{j=i}^{i+k}\hat{\phi}^{j,[\Gamma]} = \widehat{\prod_{j=i}^{i+k}\phi^{j,[\Gamma]}}.
$$ 
From these two paragraphs it follows that on the subgroup $H_+(\Sigma^\Gamma)$ we have the desired equality of semisimple characters 
$\theta_\Sigma|_{H_+(\Sigma)^\Gamma}=\theta_{\Sigma^\Gamma_{\mathrm{ext}}}=\theta_{\Sigma^\Gamma}$.
%
%
\end{proof}

\section{Restricting representations to $G^{[\Gamma]}$}\label{S:HM}

We now fix a $\Gamma$-stable truncated datum $\Sigma=(\vec{G},x,\vec{r},\vec{X},\vec{\phi})$ and complete it to a datum $(\Sigma;\sigma)$ 
as in Definition \ref{def:datum}; 
in order to simplify the notation in this section, we also denote this completed datum by $\Sigma$, understanding $\sigma$ to be fixed.
 Via \cite[\S 6.3]{MoyPrasad1996}, the point $x\in\buil(G^0)$ defines a Levi subgroup $M^0$ of $G^0$; let $M$ denote the $G$-centralizer of the centre of $M^0$, so that $M$ is a Levi subgroup of $G$. Then $(J(\Sigma),\lambda_{\Sigma})$ is a $G$-cover of $(J(\Sigma)\cap M),\lambda_{\Sigma})$, and every irreducible subquotient of $\cInd_{J(\Sigma)\cap M}^M\ \lambda_{\Sigma}$ is a supercuspidal representation of $M$, with any two such subquotients differing by a twist by some unramified character of $M$.  Let $\pi_M$ denote any such subquotient.

Given a parabolic subgroup $P$ of $G$ with Levi factor $M$, one has a finite-length parabolically induced representation $\Ind_P^G\ \pi_M\otimes 1$. Throughout this section, $\pi$ will denote an arbitrary (not necessarily irreducible) subquotient of $\Ind_P^G\ \pi_M$. The representation $\pi|_{G^{[\Gamma]}}$ is 
very often a non-semisimple representation of infinite length.  Nonetheless, in this section we prove a result which allows one to explicitly identify a number of irreducible subquotients of $\pi|_{G^{[\Gamma]}}$, up to $G^{[\Gamma]}$-inertial equivalence.  We do this by studying the decomposition of $\lambda_\Sigma|_{J(\Sigma)\cap G^{[\Gamma]}}$ into irreducible components. 

Recall that one has a decomposition $\lambda_{\Sigma}=\sigma\otimes\kappa_{\Sigma}$, with $\kappa_{\Sigma}$ being a Heisenberg--Weil lift of $\theta_\Sigma$. Since $\Sigma$ is $\Gamma$-stable, one has by Theorem~\ref{thm:simplechars} that $H_+(\Sigma)\cap G^{[\Gamma]}=H_+(\Sigma^\Gamma)$ and  $\theta_\Sigma|_{H_+(\Sigma^\Gamma)}=\theta_{\Sigma^\Gamma}$.  Our first goal is to establish an analogous relationship between the Heisenberg--Weil representations.

\begin{lemma}
The group $J(\Sigma^\Gamma)$ is contained in $J(\Sigma)\cap G^{[\Gamma]}$ with finite index.
\end{lemma}
\begin{proof}
Let $J_+(\Sigma)=G_{x,0+}^0G_{x,s_0}^1\cdots G_{x,s_{d}}^{d+1}$.  This is a $\Gamma$-invariant normal subgroup of $J(\Sigma)$ of finite index and by inductively applying Lemma~\ref{lem:thm1} we deduce that $J_+(\Sigma)^\Gamma=J_+(\Sigma^\Gamma_{\mathrm{ext}})$, which in turn is simply equal to $J_+(\Sigma^\Gamma)$.  Taking $\Gamma$-fixed points of the short exact sequence
$$
1 \to J_+(\Sigma) \to J(\Sigma) \to G^0_{x}/G^0_{x,0+}\to 1
$$
gives an exact sequence 
$$
1 \to J_+(\Sigma)^\Gamma \to J(\Sigma)^\Gamma \to (G^0_{x}/G^0_{x,0+})^\Gamma.
$$
Since $J_+(\Sigma)^\Gamma = J_+(\Sigma^\Gamma)\subseteq J(\Sigma^\Gamma)$, this yields a coarse estimate
$$
[J(\Sigma)\cap G^{[\Gamma]}:J(\Sigma^\Gamma)] \leq
[J(\Sigma)^\Gamma:J_+(\Sigma)^\Gamma] \leq \vert (G^0_{x}/G^0_{x,0+})^\Gamma \vert.
$$
\end{proof}


Note that the representation $\kappa_\Sigma$ of $J(\Sigma)$ restricts to $H_+(\Sigma)$ as a sum of copies of $\theta_\Sigma$, and so  must  restrict further to $H_+(\Sigma^\Gamma)$ as a sum of copies of $\theta_{\Sigma^\Gamma}$. In particular, any irreducible component of $\kappa_\Sigma|_{J(\Sigma^\Gamma)}$ restricts to $H_+(\Sigma^\Gamma)$ as a sum of copies of $\theta_{\Sigma^\Gamma}$.

\begin{proposition}
The restriction to $J(\Sigma^\Gamma)$ of $\kappa_\Sigma$ is $\kappa_{\Sigma^\Gamma}$-isotypic.
\end{proposition}

\begin{proof}
The representation $\kappa_\Sigma$ is defined as a tensor product $\bigotimes_{i=0}^d\kappa_\Sigma^i$ of certain representations $\kappa_\Sigma^i$.  To prove the proposition, we recall the construction of these representations following \cite{HakimMurnaghan2008} and show at the same time  that for each $i$, $\kappa_\Sigma^i|_{J(\Sigma^\Gamma)}$ is $\kappa_{\Sigma^\Gamma}^i$-isotypic. We choose notation compatible with \cite[Proposition 7.5]{LathamNevins2020}; in particularly note that the groups denoted with script $\mathcal{J}$ here are those denoted with $J$ in \cite{Yu2001}. 

For each $0\leq i\leq d$, let $J^i(\Sigma)=J(\Sigma)\cap G^i$.  Our first step is to define the groups 
\[
\mathcal{J}^{i+1}(\Sigma)=(G^i,G^{i+1})_{x,(r_i,s_i)}
\quad\text{and}\quad\mathcal{J}^{i+1}_+(\Sigma)=(G^i,G^{i+1})_{x,(r_i,s_i+)},
\]
that are ``almost complements" of $J^i(\Sigma)$ in $J^{i+1}(\Sigma)$ in the sense that for each $i$, they give the equality $J^{i+1}(\Sigma)=J^i(\Sigma)\mathcal{J}^{i+1}(\Sigma)$ and their intersection is merely $G^i_{x,r_i}$.

Choose a maximal torus $\bS$ in $\bG^{0,[\Gamma]}$ with tamely ramified splitting field $E$ such that $x\in \apart(\bG^{0,[\Gamma]},\bS,E) \cap \buil(G^0)$.  Let $\bT$ be its centralizer, which is a $\Gamma$-stable maximal torus by Lemma~\ref{lem:tori}.  Let $\Phi^i=\Phi(\bG^{i},\bT,E)$ and $\Phi^i_{\Gamma} = \Phi(\bG^{i,[\Gamma]},\bS,E)$ be the corresponding root systems.  Writing $U_\alpha$ for the root subgroup corresponding to $\alpha$ over $E$, with filtration subgroups $U_{\alpha,x,s}$, we define
\[
\mathcal{J}^{i+1}(\Sigma)(E)=\langle \bT(E)_{r_i}, U_{\alpha,x,r_i}, U_{\beta,x,s_i}\ |\ \alpha \in\Phi^{i}, \beta \in \Phi^{i+1}\setminus\Phi^i\rangle,
\]
and set $\mathcal{J}^{i+1}(\Sigma) = \mathcal{J}^{i+1}(\Sigma)(E)\cap G^{i+1}$.  The group $\mathcal{J}^{i+1}(\Sigma^\Gamma)=\mathcal{J}^{i+1}(\Sigma^\Gamma_{\mathrm{ext}})$ is defined analogously.  The compatibility of the root systems shown in Lemma~\ref{lem:tori} gives us the inclusion $\mathcal{J}^{i+1}(\Sigma^\Gamma)\subseteq \mathcal{J}^{i+1}(\Sigma)$.  
Replacing $s_i$ with $s_i+$ in the above definition gives also $\mathcal{J}_+^{i+1}(\Sigma^\Gamma)\subseteq \mathcal{J}_+^{i+1}(\Sigma)$.


The character  $\hat{\phi}^i$, restricted to $\mathcal{J}_+^{i+1}(\Sigma)\subseteq G^{i+1}_{x,s_i+}$, has kernel denoted $\mathcal{N}^{i+1}(\Sigma)$. 
Note that this kernel is uniquely determined by the image of $X^i$ in $\frak{g}^*_{x,-r_i:-r_i+}$.
Since the character $\hat{\phi}^{i,[\Gamma]}$ associated to $\Sigma^\Gamma_{\mathrm{ext}}$ is the restriction of the character $\hat{\phi}^i$ associated to $\Sigma$, we have $\mathcal{N}^{i+1}(\Sigma^\Gamma)\subset\mathcal{N}^{i+1}(\Sigma)$.

We next define
\[
\mathcal{H}_\Sigma =\mathcal{J}^{i+1}(\Sigma)/\mathcal{N}^{i+1}(\Sigma), \text{and}\;
\mathcal{Z}_\Sigma =\mathcal{J}^{i+1}_+(\Sigma)/\mathcal{N}^{i+1}(\Sigma),
\]
as well as the analogous groups defined relative to $\Sigma^\Gamma$.  The group $\mathcal{H}_\Sigma$ is a $p$-Heisenberg group with centre $\mathcal{Z}_\Sigma$. By construction, the inclusion $\mathcal{J}^{i+1}(\Sigma^\Gamma)\subseteq \mathcal{J}^{i+1}(\Sigma)$ induces an embedding $\mathcal{H}_{\Sigma^\Gamma} \to \mathcal{H}_{\Sigma}$.  Since the centres are both isomorphic to $\mathbb{F}_p$, 
the map $\mathcal{Z}_{\Sigma^\Gamma}\to \mathcal{Z}_\Sigma$ induced by the embedding is an isomorphism.  

Since $\hat{\phi}^i$ gives a (non-trivial) character of $\mathcal{Z}_\Sigma\cong \mathcal{Z}_{\Sigma^\Gamma}$, the Stone--von Neumann theorem associates to $\hat{\phi}^i$ both a unique irreducible representation $\eta_\Sigma^i$ of $\mathcal{H}_\Sigma$ and a unique irreducible representation $\eta_{\Sigma^\Gamma}^i$ of $\mathcal{H}_{\Sigma^\Gamma}$. 
In particular, any irreducible component of $\eta_\Sigma^i|_{\mathcal{H}_{\Sigma^\Gamma}}$ is an irreducible representation of $\mathcal{H}_{\Sigma^\Gamma}$ with central character $\hat{\phi}^i$; by uniqueness,  
the restriction $\eta_\Sigma^i|_{\mathcal{H}_{\Sigma^\Gamma}}$ must be $\eta_{\Sigma^\Gamma}^i$-isotypic.

Next, we define the group $\mathcal{W}_\Sigma=\mathcal{J}^{i+1}(\Sigma)/\mathcal{J}_+^{i+1}(\Sigma)$. When nontrivial, this group naturally carries the structure of an $\frak{f}$-vector space, on which the character $\hat{\phi}^i$ defines a symplectic form by composing with the commutator. 
Carrying out the analogous construction with $\Sigma^\Gamma$ yields an embedding
$\mathcal{W}_{\Sigma^\Gamma} \to \mathcal{W}_\Sigma$ that respects these symplectic structures. 
As shown carefully in \cite[\S10]{Yu2001}, via the conjugation action of $J^i(\Sigma)$ on $\mathcal{J}^{i+1}(\Sigma)$ we may by the theory of the Heisenberg--Weil lift extend $\eta_\Sigma^i$ to a representation $\hat{\eta}_\Sigma^i$ of $J^i(\Sigma)\ltimes \mathcal{J}^{i+1}(\Sigma)$, and similarly for $\Sigma^\Gamma$.  
Moreover, up to a minor choice (that can be made consistently) when $p=3$ and $\dim(\mathcal{W}_\Sigma)=2$ \cite[\S2.3]{HakimMurnaghan2008}, this extension is unique.  
It therefore follows that the restriction of $\hat{\eta}_\Sigma^i$ to $J^i(\Sigma^\Gamma)\ltimes \mathcal{J}^{i+1}(\Sigma^\Gamma)$ is $\hat{\eta}_{\Sigma^\Gamma}^i$-isotypic.

Note that when $G^{i,[\Gamma]}=G^{i+1,[\Gamma]}$ we simply have $\mathcal{J}^{i+1}(\Sigma^\Gamma)=\mathcal{J}^{i+1}_+(\Sigma^\Gamma)=G^{i,[\Gamma]}_{x,r_i}$, so that $\mathcal{H}_{\Sigma^\Gamma}=\mathcal{Z}_{\Sigma^\Gamma}$, $\eta^i_{\Sigma^\Gamma}=\hat{\phi}^i$ and $\mathcal{W}_{\Sigma^\Gamma}$ is trivial.  In this case,  $\eta^i_{\Sigma^\Gamma}=1\ltimes \phi^{i,[\Gamma]}|_{G^{i,[\Gamma]}_{x,r_i}}$ and $\mathcal{J}^{i+1}(\Sigma^\Gamma) \subset J^i(\Sigma^\Gamma)=J^{i+1}(\Sigma^\Gamma)$.

Finally, the representation $\kappa_\Sigma^i$ is defined by making the identification $J^{i+1}(\Sigma)=J^i(\Sigma)\mathcal{J}^{i+1}(\Sigma)$ in order to define
\[\kappa_\Sigma^i(gj)=\phi^i(g)\hat{\eta}^i_\Sigma(g,j),
\]
for $g\in J^i(\Sigma)$ and $j\in\mathcal{J}^{i+1}(\Sigma)$. This representation may then be inflated to the group $J(\Sigma)$.  Following the same process for $\Sigma^{\Gamma}$ yields the representation $\kappa^i_{\Sigma^\Gamma}$, and we deduce 
that the restriction of $\kappa_\Sigma^i$ is $\kappa_{\Sigma^\Gamma}^i$-isotypic, as required. 
\end{proof}

From this, one immediately deduces the following.

\begin{corollary}\label{cor:xi}
Any irreducible component of $(\sigma\otimes\kappa_\Sigma)|_{J(\Sigma^\Gamma)}$ is isomorphic to $\xi\otimes\kappa_{\Sigma^\Gamma}$ for some irreducible component $\xi$ of $\sigma|_{J(\Sigma^\Gamma)}$.
\end{corollary}

We therefore have achieved a decomposition into irreducible components of $\lambda_\Sigma|_{J(\Sigma^\Gamma)}$. Let us now apply this to the study of $\pi|_{G^{[\Gamma]}}$.

If the representation $\xi$ is an irreducible component of $\sigma|_{J(\Sigma^\Gamma)}$ that is cuspidal when viewed as a representation of $G_{x,0:0+}^{0,[\Gamma]}$, then $\xi\otimes\kappa_{\Sigma^\Gamma}$ is an $\frak{s}$-type, for $\frak{s}=\frak{s}_\xi$ the inertial equivalence class in $G^{[\Gamma]}$ of a supercuspidal representation of $M^{[\Gamma]}$, the Levi subgroup corresponding to $\Sigma^\Gamma$ \cite[Thm. 7.5]{KimYu2017}.  


If not, then it is not initially clear which inertial equivalence classes $\xi\otimes\kappa_{\Sigma^\Gamma}$ may intertwine with. In order to deal with this issue, we carry out a process established in \cite[7.2--7.5]{LathamNevins2020} in order to construct, for each component $\xi$, an open subgroup $J_\xi$ of $J(\Sigma^\Gamma)$ and an irreducible component $\lambda_\xi$ of $\xi\otimes\kappa_{\Sigma^\Gamma}|_{J_\xi}$ such that $(J_\xi,\lambda_\xi)$ is an $\frak{s}_\xi$-type for some $G^{[\Gamma]}$-inertial equivalence class $\frak{s}_\xi$. Rather than repeat the arguments of \cite[7.2--7.5]{LathamNevins2020} \emph{mutatis mutandis}, we simply outline the key steps of the construction: 
\begin{itemize}
\item The representation $\xi$ is an irreducible representation of $J(\Sigma^\Gamma)$ that is trivial on $J_+(\Sigma^\Gamma)$, and therefore descends to an irreducible representation of $J(\Sigma^\Gamma)/J_+(\Sigma^\Gamma)\simeq G_{x}^{0,[\Gamma]}/G_{x,0+}^{0,[\Gamma]}$. By Clifford theory, further restriction gives a sum of pairwise $G_{x}^{0,[\Gamma]}$-conjugate irreducible representations of $G_{x,0:0+}^{0,[\Gamma]}$, which is a finite group of Lie type. Fix an element $\xi'$ of this $G_{x}^{0,[\Gamma]}$-orbit, and let $(\mathsf{L}_\xi',\zeta_\xi')$ denote a representative of its cuspidal support.
\item Using \cite[Lem. 7.2]{LathamNevins2020}, we associate to $\mathsf{L}_\xi'$ a point $u_\xi\in\buil(M^{0,[\Gamma]})$. The point $u_\xi$ in turn defines a Levi subgroup of $G^{0,[\Gamma]}$, which we denote by $L_\xi^0$.
\item Via the inclusion $G_{u_\xi}^{0,[\Gamma]}\hookrightarrow G_x^{0,[\Gamma]}$, one may choose an irreducible component $\zeta_\xi$ of $\xi|_{G_{u_\xi}^{0,[\Gamma]}}$ such that the restriction of $\zeta_\xi$ to $G_{u_\xi,0}^{0,[\Gamma]}$ is a sum of $G_{u_\xi}^{0,[\Gamma]}$-conjugates of $\zeta_\xi'$.
\item Set $\lambda_{\Sigma_\xi^\Gamma}=\zeta_\xi\otimes\kappa_{\Sigma^\Gamma}$. It then follows from \cite[Lem. 7.3]{LathamNevins2020} that the tuple $(\Sigma_\xi^\Gamma;\zeta_\xi)$ derived from $(\Sigma^\Gamma;\sigma)$ by replacing $x$ with $u_\xi$ and $\sigma$ with $\zeta_\xi$ 
is a datum with corresponding type $(J(\Sigma_\xi^\Gamma),\lambda_{\Sigma_\xi^\Gamma})$.
Moreover, \cite[Lem. 7.2]{LathamNevins2020} shows that $J(\Sigma_\xi^\Gamma)\subset J(\Sigma^\Gamma)$, and it follows from \cite[Prop. 7.5]{LathamNevins2020} that
\[\Hom_{J(\Sigma_\xi^\Gamma)}(\lambda_{\Sigma_\xi^\Gamma},\ \xi\otimes\kappa_{\Sigma^\Gamma})\neq 0.
\]
\end{itemize}
Let $\frak{s}_\xi$ denote the $G^{[\Gamma]}$-inertial equivalence class  such that $(J(\Sigma_\xi^\Gamma),\lambda_{\Sigma_\xi^\Gamma})$ is an $\frak{s}_\xi$-type.

\begin{theorem}\label{thm:main}
Let $\Sigma$ and $\pi$ be as above. For each irreducible component $\xi$ of $\sigma|_{J(\Sigma^\Gamma)}$, there exists an irreducible subquotient $\pi_\xi$ of $\pi|_{G^{[\Gamma]}}$ that has inertial support $\frak{s}_\xi$.
\end{theorem}

\begin{proof}
Since $\pi$ contains $\lambda_{\Sigma}$, it contains a copy of $\lambda_{\Sigma_\xi^{\Gamma}}$ for each representation $\xi$, which implies that $\pi|_{G^{[\Gamma]}}$ has a non-trivial projection onto $\Rep_{\lambda_{\Sigma_\xi^\Gamma}}(G^{[\Gamma]})$, the category of smooth representations generated by their $\lambda_{\Sigma_\xi^\Gamma}$-isotypic vectors, for each $\xi$. As the pair $(J(\Sigma_\xi^\Gamma),\lambda_{\Sigma_\xi^\Gamma})$ is an $\frak{s}_\xi$-type, one has $\Rep_{\lambda_{\Sigma_\xi^\Gamma}}(G^{[\Gamma]})=\Rep^{\frak{s}_\xi}(G^{[\Gamma]})$, the corresponding Bernstein block. The representation $\pi$ therefore has non-trivial projection onto this block, and the image of $\pi$ under this projection, being finitely generated, admits an irreducible quotient. By construction, any such irreducible quotient has inertial support $\frak{s}_\xi$.
\end{proof}

\begin{remark}
In general, this does not provide an exhaustive list of the inertial supports of representations occuring in $\pi|_{G^{[\Gamma]}}$. The finite-length representation $\pi$ is generated by its $\lambda_{\Sigma}$-isotypic vectors, hence by its vectors on which $J(\Sigma^\Gamma)$ acts as $\xi\otimes\kappa_{\Sigma^\Gamma}$ for some $\xi$, but only as a representation of $G$. In general, one may apply Mackey's theorem to say that
\[\pi|_{G^{[\Gamma]}}\hookrightarrow\Res_{G^{[\Gamma]}}^G\cInd_{J(\Sigma)}^G\ \lambda_\Sigma = \bigoplus_{G^{[\Gamma]}\backslash G/J(\Sigma)}\cInd_{^gJ(\Sigma)\cap G^{[\Gamma]}}^{G^{[\Gamma]}}\ ^g\lambda_{\Sigma}.
\]
What we have done is describe (exhaustively) the inertial equivalence classes of those subrepresentations of $\pi|_{G^{[\Gamma]}}$ which are contained within the trivial Mackey summand
\[\cInd_{J(\Sigma)\cap G^{[\Gamma]}}^{G^{[\Gamma]}}\ \lambda_{\Sigma}.
\]
However, the other summands are far less amenable to description.   For example, while it is true that for any $g\in G$, conjugation of the pair $(J(\Sigma),\lambda_{\Sigma})$ gives the pair $(J({}^g\Sigma),\lambda_{{}^g\tilde{\Sigma}})$, where $^g\Sigma$ is still a datum for $G$, the datum $^g\Sigma$ need not be $\Gamma$-stable. In any situations where $^g\Sigma$ \emph{is} $\Gamma$-stable, then we may similarly describe the inertial equivalence classes of representations occurring within the Mackey summand
\[\cInd_{^gJ(\Sigma)\cap G^{[\Gamma]}}^{G^{[\Gamma]}}\ ^g\lambda_{\Sigma}.
\]
One expects, however, that many of these summands do not correspond to $\Gamma$-stable data.
\end{remark}

\providecommand{\bysame}{\leavevmode\hbox to3em{\hrulefill}\thinspace}
\providecommand{\MR}{\relax\ifhmode\unskip\space\fi MR }
\providecommand{\MRhref}[2]{%
  \href{http://www.ams.org/mathscinet-getitem?mr=#1}{#2}
}
\providecommand{\href}[2]{#2}

\end{document}